\documentclass{llncs}
\usepackage{graphicx}
\usepackage{amsmath}
\usepackage{amssymb}
\usepackage{enumerate}
\usepackage{varioref}
\usepackage{charter,eulervm}

\newcommand{\es}{\varnothing}

\title{\sc {Some Results on Triangle Partitions}}

\author{
 Ton~Kloks%
\thanks{This author is supported by the
National Science Council of Taiwan, under grant
NSC~99--2218--E--007--016.}
\and
 Sheung-Hung~Poon
}
\institute{
 Department of Computer Science\\
 National Tsing Hua University,
 No.~101, Sec.~2, Kuang Fu Rd., Hsinchu, Taiwan\\
 {\tt spoon@cs.nthu.edu.tw}
}

\begin{document}

\maketitle

\begin{abstract}
We show that there exist efficient algorithms for
the triangle packing problem in colored permutation graphs,
complete multipartite graphs, distance-hereditary graphs,
$k$-modular permutation graphs and complements of
$k$-partite graphs (when $k$ is fixed).
We show that there is an efficient
algorithm for $C_4$-packing on bipartite permutation graphs and we
show that $C_4$-packing on bipartite graphs is NP-complete. We
characterize the cobipartite graphs that have a triangle partition.
\end{abstract}

\section{Introduction}

A triangle packing in a graph $G$ is a collection of vertex-disjoint
triangles. The triangle packing problem asks for a triangle packing
of maximal cardinality. The triangle partition problem asks whether
the vertices of a graph can be partitioned into triangles.
We refer to Appendix~\ref{prel TP} for an overview of known results
on triangle packing problems.

Our objective is the study of the triangle partition
problem on permutation graphs. We establish polynomial-time
algorithms for several classes of graphs that are related to
cographs and permutation graphs. We show that there exist
polynomial-time algorithm for triangle  packings of
colored permutations, complete multipartite graphs, $k$-modular
permutation graphs, distance-hereditary graphs and complements of
$k$-partite graphs. We show that the $C_4$-packing problem can be solved
on bipartite permutation graphs and that this problem becomes NP-complete
on the class of bipartite graphs. We also characterize the cobipartite
graphs that admit a triangle partition.

Since a lot of research on triangle packing centers on interval 
graphs, we start our discussion with this class of graphs. 
 
\section{Interval graphs}

A graph is an interval graph if it is the intersection
graph of a collection of intervals on the real line~\cite{kn:hajos}.

A consecutive clique arrangement of a graph $G$ is a linear arrangement
of its maximal cliques such that for each vertex, the cliques that contain
it are consecutive.
Let $G$ be an interval graph with $n$ vertices.
Then it has at most $n$ maximal cliques.
The following theorem was proved in~\cite[Theorem~7.1]{kn:fulkerson}.

\begin{theorem}[\cite{kn:fulkerson}]
A graph $G$ is an interval graph if and only if it has a
consecutive clique arrangement.
\end{theorem}

Consider the following problem, called
the `partition into bounded cliques' problem in~\cite{kn:bodlaender2}.
\begin{quote}
Let $G=(V,E)$ be a graph and let
$r$ and $s$ be integers. Can $V$ be partitioned into
$s$ cliques each of cardinality at most $r$?
\end{quote}
This problem can be solved in linear time on interval
graphs~\cite{kn:bodlaender2,kn:papadimitriou}.

\smallskip
In the following we denote by $\tau(G)$ the maximal number of
vertex disjoint triangles in $G$.
The following lemma is easy to check.

\begin{lemma}
\label{at least three vertices}
Let $G$ be an interval graph and let $[C_1,\ldots,C_t]$ be a
consecutive clique arrangement.
\begin{enumerate}[\rm (i)]
\item If there is a maximal clique
with only one vertex $x$ then $x$ is an isolated vertex.
In that
case $\tau(G)=\tau(G-x)$.
\item
If there is a maximal clique with exactly two vertices $x$ and $y$
then $(x,y)$ is a bridge.
Let $G^{\prime}$ be the graph obtained from $G$ by deleting the edge
$(x,y)$ but not its endvertices. Then $G^{\prime}$ is an interval
graph and $\tau(G)=\tau(G^{\prime})$.
\end{enumerate}
\end{lemma}

\noindent
Henceforth we may assume that every maximal clique in the
consecutive clique arrangement has at least three vertices.

\begin{lemma}
\label{greedy}
Let $G$ be an interval graph and let $[C_1,\ldots,C_t]$
be a consecutive arrangement of its maximal cliques. Consider
an ordering of the vertices in $C_1$ by increasing degree or,
equivalently, by increasing right endpoints of the corresponding
intervals.
There is a triangle packing of $G$ with maximal cardinality
such that all vertices of $C_1$ are covered except, possibly the
smallest or, the two smallest vertices.
\end{lemma}
\begin{proof}
If there are at least three vertices in $C_1$ not covered by a triangle
then we can add a triangle to the packing.

Let $\alpha$ and $\beta$ be two vertices in $C_1$ and assume that
$\alpha < \beta$ in the right endpoint ordering. Assume that $\alpha$ is
in a triangle of a triangle packing $\mathcal{P}$ and that
$\beta$ is not. Then we can switch $\alpha$ and $\beta$ in the triangle
and obtain an alternative packing $\mathcal{P}^{\prime}$ with the
same number of triangles such that $\beta$ is covered. The claim follows
by induction by recursively switching any vertex that is not
covered with the
smallest vertex in $C_1$ that is covered.
\qed\end{proof}

\begin{lemma}
\label{all three}
Let $G$ be an interval graph and let
$[C_1,\ldots,C_t]$ be a consecutive arrangement
of its maximal cliques. Let $\alpha$, $\beta$ and $\gamma$ be the three
smallest vertices of $C_1$ in the ordering by right endpoints.
If $\alpha$ is covered by a triangle in a triangle packing
$\mathcal{P}$ then there exists a triangle packing
$\mathcal{P}^{\prime}$ of the same cardinality as $\mathcal{P}$
such that $\{\alpha,\beta,\gamma\}$ is a triangle of $\mathcal{P}^{\prime}$.
\end{lemma}
\begin{proof}
Since $C_1$ is a maximal clique it has a vertex that is not contained
in $C_2$. Consequently, $\alpha$ is contained only in $C_1$.
Consider a triangle $\{\alpha,p,q\}$ in $\mathcal{P}$. Then
$p$ and $q$ are in $C_1$. Let $p$ be the smallest of the two.
Assume that
$\beta \neq p$. If $\beta$ is not covered by any triangle of
$\mathcal{P}$ then we can replace $p$ with $\beta$. Assume that
$\beta$ is in a triangle $\{\beta,r,s\}$ of $\mathcal{P}$. Let $C_i$
be the first clique that contains all three vertices $\beta$, $r$ and $s$.
Then $C_i$ also contains $p$, since $p$ is larger than $\beta$.
Replace the two triangles $\{\alpha,p,q\}$ and $\{\beta,r,s\}$
with $\{\alpha,\beta,q\}$ and $\{p,r,s\}$. A similar argument
shows that, if $q \neq \gamma$ then we can replace $q$ with $\gamma$
and obtain an alternative packing $\mathcal{P}^{\prime}$ with the
same number of triangles.
\qed\end{proof}

\begin{theorem}
The triangle partition problem can be solved in linear time
on interval graphs. The triangle packing problem can be solved
by an exponential algorithm which runs in $O^{\ast}(1.47^n)$ time.
\end{theorem}
\begin{proof}
The first claim follows from Lemma~\ref{all three}.
The second claim follows from the recurrence
\[T(n)=T(n-1)+T(n-3).\]
To see that this recurrence holds, observe
that the minimal element of $C_1$ is either not in any
triangle or it is in a triangle together with the next two
smallest elements of $C_1$.
\qed\end{proof}

An $O(n \log n)$
algorithm for maximum matching in interval graphs is presented
in~\cite{kn:liang,kn:moitra}. For the class of strongly chordal graphs,
which includes the
class of interval graphs, there exists
a linear-time algorithm for maximum matching
when a strong elimination ordering of the
graph is part of the input~\cite{kn:dahlhaus}.
Dahlhaus {\em et al.\/}, extend the greedy algorithm for
a $K_r$-partition for general $r$
on interval graphs
to the class of strongly chordal graphs~\cite{kn:dahlhaus}.

\smallskip
Concerning packings of vertex-disjoint maximal cliques
we have the following theorem.

\begin{theorem}
Let $G$ be an interval graph. There exists a
linear-time algorithm that finds the maximal number of
vertex-disjoint maximal cliques.
\end{theorem}
\begin{proof}
This can be seen as follows. First recall
that the maximal number of vertex-disjoint maximal cliques in an
interval graph is equal to the minimal number of vertices that
represent all maximal cliques~\cite{kn:fulkerson}. A vertex
$x$ represents a clique $C$ if $x \in C$. A set of
vertices that together represent all maximal cliques is called
a clique-transversal.

Let $[C_1,\ldots,C_t]$ be a consecutive clique ordering of $G$.
We use the following
trick that we learned from~\cite{kn:guruswami2} to reduce the
problem to a domination problem.
Add one vertex to each maximal clique $C_i$. Then the new graph
$H$ has $n+t$ vertices and
$H$ is an interval graph.
It is easy to check that the minimal cardinality
of a clique transversal in $G$ is equal to the minimal cardinality of a
dominating set in $H$. Here, a dominating set in a graph is a
set $S$ of vertices such that every vertex not in $S$ has a neighbor
in $S$. There exists a linear-time algorithm that finds a
dominating set of minimal cardinality in an interval
graph~\cite{kn:kellog}. This proves the claim.
\qed\end{proof}

\section{Triangle partition on colored permutation graphs}

We refer to Appendix~\ref{prel perm} for a brief overview on
permutation graphs.

Notice that the triangle partition problem for permutation graphs
is equivalent to the following problem.
\begin{quote}
Let $\pi$ be a
permutation of $V=\{1,\ldots,n\}$. Decide if there exist a partition
$\mathcal{P}$
of $\{1,\ldots,n\}$ into triples such that for each element
$\{i,j,k\} \in \mathcal{P}$ with $i < j < k$,
\[\pi(i) < \pi(j) < \pi(k).\]
\end{quote}
As far as we know, both the triangle partition - and the
triangle packing problem for permutation
graphs are open. In this section we consider a variation of the
partition problem.

\begin{theorem}
Let $\pi \in S_n$ and let
\[c: \{1,\ldots,n\} \rightarrow \{1,2,3\}. \]
There exists a polynomial-time algorithm that decides whether
there exists a partition $\mathcal{P}$ of $\{1,\ldots,n\}$ into
triples such that for each $\{i,j,k\} \in \mathcal{P}$ with
$i < j< k$:
\[\pi(i) < \pi(j) < \pi(k) \quad\text{and}\quad
c(\pi(i))=1, \;\; c(\pi(2))=2 \;\;\text{and}\;\; c(\pi(3))=3.\]
\end{theorem}
\begin{proof}
Construct two bipartite graphs as follows. The first bipartite graph
has vertices that are the elements with colors 1 and 2. A vertex $i$
with color 1 is adjacent to a vertex $j$ with color 2 if
\[ i < j \quad\text{and}\quad \pi(i) < \pi(j).\]
The second bipartite graph has vertices with
colors 2 and colors 3. A vertex $p$ with color 2 is adjacent to a
vertex $q$ with color 3 if
\[ p < q \quad\text{and}\quad \pi(p) < \pi(q).\]
It is easy to check that there exists a partition into triangles
if and only if both bipartite graphs have a perfect matching.

One can find a perfect matching in $O(n^{\omega})$ time~\cite{kn:mucha},
where $\omega$ is the exponent of a matrix multiplication algorithm.
\qed\end{proof}

It is easy to check that this result generalizes to the
$K_r$-partitioning problem when an $r$-coloring is a part of the input.

\section{Triangle partition on complete multipartite graphs}

Complete $t$-partite graphs form a subclass of the permutation graphs.
In this section we show that the triangle partition problem can be
solved in polynomial time for complete $t$-partite graphs.

Let $a_1,\ldots,a_t$ be positive natural numbers.
We use $K(a_1,\ldots,a_t)$ to denote the complete
$t$-partite graph with  color classes $A_1,\ldots,A_t$
such that $|A_i|=a_i$ for all $i \in \{1,\ldots,t\}$.

\begin{lemma}
\label{large classes}
Let $G$ be a $t$-partite graph. If there exists a triangle partition
$\mathcal{P}$ of $G$ then there exists a triangle partition
$\mathcal{P}^{\prime}$ which contains a triangle with vertices in
three of the largest color classes.
\end{lemma}
\begin{proof}
Let $A_1,\ldots,A_t$ be the color classes of $G=(V,E)$ and let
$a_i=|A_i|$ such that
\[0< a_1\leq a_2 \leq \ldots \leq a_t.\]
We prove the claim by induction on the number of vertices.
If there are only three vertices then the claim is obviously true.
Consider a triangle $P \in \mathcal{P}$
and consider the subgraph $G^{\prime}$
of $G$ induced by $V-P$. Then $G^{\prime}$ has a triangle partition
$\mathcal{P}-P$.
By induction there exists a triangle partition $\mathcal{Q}$
of $G^{\prime}$
with a triangle $Q$ contained in three of the maximal color classes
of $G^{\prime}$. Note that color classes with the same cardinality are
interchangeable. Therefore, we may assume that the three maximal classes
of $G^{\prime}$ that contain the vertices of $Q$ are also maximal
classes of $G$.
\qed\end{proof}

\begin{theorem}
There exists a linear-time algorithm that solves the
triangle partition problem on
complete $t$-partite graphs.
\end{theorem}
\begin{proof}
By Lemma~\ref{large classes} a greedy algorithm which chooses
recursively triangles from three of the largest color classes
produces a triangle partition if it exists. It is easy to see
that this algorithm can be implemented to run in linear time.
\qed\end{proof}

\begin{remark}
We have not found an easy condition on the numbers $a_1,\ldots,a_t$
which characterizes the complete $t$-partite graphs that have a
triangle partition.
\end{remark}

\begin{remark}
A similar greedy algorithm solves the triangle packing problem on
complete $t$-partite graphs.
\end{remark}

\section{$C_4$-Packing on bipartite permutation graphs}

Let $G=(X,Y,E)$ be a bipartite permutation graph and consider
the left-to-right orderings of the vertices of $X$ and $Y$ on
the topline of a permutation diagram. It is easy to check
that this is a strong ordering, which is defined as follows.

\begin{definition}
Let $G=(X,Y,E)$ be a bipartite graph. A {\em strong ordering\/}
is a pair of linear orderings $<_1$ and $<_2$ on $X$ and $Y$
such that for all $x_1,x_2 \in X$ and $y_1,y_2 \in Y$ with
$x_1 <_1 x_2$ and $y_1 <_2 y_2$
\[\left( (x_1,y_2) \in E \;\;\text{and}\;\;(x_2,y_1) \in E \right)
\;\;\Rightarrow\;\;
\left((x_1,y_1) \in E \;\;\text{and}\;\;(x_2,y_2) \in E\right).\]
\end{definition}

Spinrad, {\em et al.\/}, obtained
the following characterization~\cite{kn:spinrad2}.

\begin{theorem}[\cite{kn:spinrad2}]
Let $G=(X,Y,E)$ be a bipartite graph. Then $G$ is a bipartite
permutation graph if and only if there is a strong ordering on
$X$ and $Y$.
\end{theorem}

In this section we show that there is a greedy algorithm that
computes a $C_4$-packing on bipartite permutation graphs.
Consider a diagram for a bipartite permutation graph $G=(X,Y,E)$
and let $<$ be the left-to-right ordering of the points on the
topline. Denote by $<_1$ and $<_2$ the sub-orderings of $<$
induced by the vertices of $X$ and $Y$.

\begin{lemma}
\label{right}
Let
$\mathcal{P}$ be a $C_4$-packing on $G$. Assume that $a,c \in X$
with $a <_1 c$ and assume that $a$ and $c$ are in
a square $C \in \mathcal{P}$.
Let $a <_1 b <_1 c$ and assume that
$b$ is the smallest element $>_1 a$.
Then there is a packing $\mathcal{P}^{\prime}$
of the same cardinality as $\mathcal{P}$ such that $a$ and $b$ are in
a square $C^{\prime} \in \mathcal{P}^{\prime}$.
\end{lemma}
\begin{proof}
Let $C=\{a,c,p,q\}$.
First notice that for each vertex $y \in Y$, its neighborhood $N(y)$
forms an interval in $(X,<_1)$~\cite{kn:spinrad2}. Then
\[p,q \in N(a) \cap N(c) \;\;\text{and}\;\; a <_1 b <_1 c
\quad\Rightarrow\quad p,q \in N(a) \cap N(b).\]
Thus $\{a,b,p,q\}$ is a square. If $b$ is not in any square of $\mathcal{P}$
then we can replace $C$ with $C^{\prime}=\{a,b,p,q\}$.
Assume that $b$ is in a square $C_2=\{b,d,r,s\} \in \mathcal{P}$.
We consider the following cases.

First assume that $a <_1 b <_1 d <_1 c$.
Then $p$ and $q$ are adjacent to
$a$, $b$, $c$ and $d$.
Since $r$ and $s$ intersect the linesegments of $b$ and $d$ but not the
linesegments of $p$ and $q$ each of $r$ and $s$ intersects at least one of
$a$ and $c$. If $r$ and $s$ both intersect $a$ then we can replace
$C$ and $C_2$ with $\{a,b,r,s\}$ and $\{c,d,p,q\}$. Similarly, of
both $r$ and $s$ intersect $c$ we can replace $C$ and $C_2$ with
$\{a,b,p,q\}$ and $\{c,d,r,s\}$. Assume that $r$ intersects
$a$ and that $s$ intersects $c$. Then replace $C$ and $C_2$ with
$\{a,b,p,r\}$ and $\{c,d,q,s\}$.

Now assume that $a <_1 b <_1 c <_1 d$.
Then
\[ b <_1 c <_1 d \;\;\text{and}\;\; r,s \in N(b) \cap N(d)
\quad \Rightarrow \quad r,s \in N(c) \cap N(d).\]
Thus $\{c,d,r,s\}$ is a square. Replace
$C$ and $C_2$ with $\{a,b,p,q\}$ and $\{c,d,r,s\}$.

This proves the lemma.
\qed\end{proof}

By Lemma~\ref{right}, there
exists a maximum $C_4$-packing $\mathcal{P}$
of $G$ such that the pairs of vertices of
$X$ that are contained in a square of $\mathcal{P}$ are
consecutive
in $<_1$ and the pairs of vertices of $Y$ that are contained in a
square of $\mathcal{P}$ are consecutive in $<_2$.

\begin{lemma}
\label{left}
Consider two squares $C_1=\{a,b,r,s\}$ and $C_2=\{c,d,p,q\}$
in a $C_4$-packing $\mathcal {P}$.
Assume that $a < b < c < d$ and that $p < q < r < s$.
There exists a packing $\mathcal{P}^{\prime}$ of the
same cardinality as $\mathcal{P}$ such that $\{a,b,p,q\}$ and
$\{c,d,r,s\}$ are squares in $\mathcal{P}^{\prime}$.
\end{lemma}
\begin{proof}
Assume that $d<p$.
Since the linesegments of $p$ and $q$ intersect the linesegments
of $c$ and $d$, and $r$ and $s$ intersect $a$ and $b$, and since
the linesegments of $p$, $q$, $r$ and $s$ are parallel, by the ordering
the linesegments of $p$, $q$, $r$ and $s$ intersect all linesegments
of $a$, $b$, $c$ and $d$. Thus $\{c,d,r,s\}$ and $\{a,b,p,q\}$ are
squares. The only other possible case is where $s < a$. This case is
similar.

This proves the lemma.
\qed\end{proof}

\begin{theorem}
There exists a linear-time algorithm which computes a maximum
$C_4$-packing in a bipartite permutation graph.
\end{theorem}
\begin{proof}
Let $G=(X,Y,E)$ be a bipartite permutation graph and
let $<$ the the ordering of the vertices on the topline of a
diagram for $G$. We prove that there exists a
maximum packing $\mathcal{P}$
such that the first four vertices in the ordering $<$ that form a
square are in $\mathcal{P}$. Consider the first four
vertices $C=\{x_1,x_2,y_1,y_2\}$ that form a square. We may assume
that $x_1<x_2<y_1<y_2$. Assume that $C \not\in \mathcal{P}$.
Consider the square
$C_1=\{x_1^{\prime},x_2^{\prime},y_1^{\prime},y_2^{\prime}\}$
in $\mathcal{P}$ with the smallest
vertices in $X$.
If $C \cap C_1 =\es$ then $C$ is disjoint from all squares
in $\mathcal{P}$ which contradicts the maximality of $\mathcal{P}$.
In all other cases we can replace $C_1$ with $C$.

This proves the correctness of the following algorithm.
Remove vertices that are smallest in the $<$-ordering that are
not in any square of $G$. Assume next that the smallest element
is $x_1 \in X$. Thus $x_1$ is in a square. Take the first
element $x_2 \in X$ that is in a square with $x_1$. Then take the
first two elements $y_1,y_2 \in Y$ such that $C=\{x_1,x_2,y_1,y_2\}$
induces a square in $G$. Put $C$ in $\mathcal{P}$. Remove the
vertices $C$ from $G$ and recurse. It is easy to see that,
with some care this algorithm
can be implemented to run in linear time.
\qed\end{proof}

\section{NP-Completeness for $C_4$-packing on bipartite graphs}

In the previous section we proved that the $C_4$-packing problem on
bipartite permutation graphs can be solved in linear time.
In this section we show that the $C_4$-packing problem is 
NP-complete for general bipartite graphs.

\begin{theorem}
\label{NP-c}
$C_4$-Packing on bipartite graphs is NP-complete.
\end{theorem}
\begin{proof}
It is easy to see that the problem is in NP.
To show the NP-hardness for our problem we use a reduction
from the 3-dimensional matching problem (3DM) which is described
as follows. 
Suppose we are given three sets $X, Y$ and $Z$
such that $|X|=|Y|=|Z|=q$, and a set $M \subseteq X \times Y
\times Z$ of triples $(x,y,z)$. The 3DM problem asks for a subset
$M^{\prime}$
of $M$ such that each element of $X, Y$ and $Z$ is contained
in exactly one triple in $M^{\prime}$.

We apply local replacements to the input instance of 3DM.
For each pair $x$ and $y$ that appear in some triple $(x,y,z) \in M$,
we create a vertex $v_{xy}$ and a path of length two 
$[x,v_{xy},y]$ as shown
in Figure~\ref{fig:replace}.

\begin{figure}
  \centering
  \includegraphics[width=.35\columnwidth]{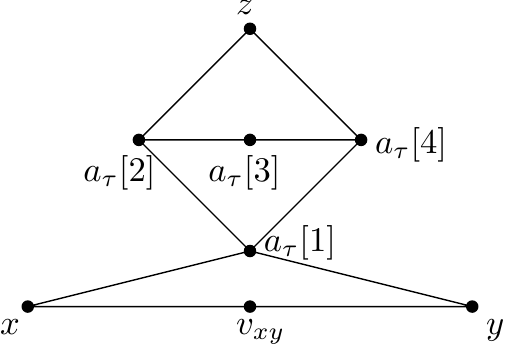}
  \caption{Local replacement for triple $\tau = (x,y,z)$ in $M$.}
  \label{fig:replace}
\end{figure}

For each triple $\tau=(x,y,z) \in M$, we create
four local vertices 
\[a_{\tau}[i] \quad 1\leq i \leq 4\] 
and eight edges
as shown in Figure~\ref{fig:replace}. 
This finishes the description of
the construction of the bipartite graph $G$. The bipartite graph $G$   
is obtained in linear time.

\smallskip 
Suppose that $M^{\prime}$ is a solution of the 3DM problem.
We obtain a $C_4$-packing of size $p = q + |M|$ as follows.
The packing is constructed by taking 4-cycles
\[\begin{cases}
\{x, v_{xy}, y, a_{\tau}[1]\} \;\;\text{and}\;\;
\{a_{\tau}[2],a_{\tau}[3], a_{\tau}[4],z\}
& \quad\text{if}\quad \tau=(x,y,z) \in M^{\prime}\\
\{a_{\tau}[1], a_{\tau}[2], a_{\tau}[3], a_{\tau}[4]\}
& \quad\text{if}\quad \tau \not\in M^{\prime}. 
\end{cases}\]
This ensures that each element in $X \cup Y \cup Z$ is included in exactly
one 4-cycle in the packing.

Now assume that there is a $C_4$-packing of size $p$, 
where $p = q + |M|$.
In the constructed figure for any triple $\tau=(x,y,z) \in M$,
there are only two possible ways to pack the 4-cycles:
one way contains the 4-cycle induced by 
\[\{a_{\tau}[1], a_{\tau}[2], a_{\tau}[3], a_{\tau}[4] \}\] 
and
the other way contains the two 4-cycles induced by
\[\{x, v_{xy}, y, a_{\tau}[1] \}
\quad\text{and}\quad 
\{a_{\tau}[2], a_{\tau}[3], a_{\tau}[4], z \}.\]
The first choice contains none of vertices from $X \cup Y \cup Z$
and the second contains exactly one vertex from each of
$X, Y$ and $Z$.
Our packing is of size $p=q+|M|$ so there must be at least $q$ triples
that use the second kind of $C_4$'s in the packing.
Since none of these $q$ triples have a common element, they cover
all $3q$ elements in $X \cup Y \cup Z$. Thus these $q$ triples forms a
3-dimensional matching.
\qed\end{proof}

\section{Triangle partition on cobipartite graphs}

A cobipartite graph is the complement of a bipartite graph.
We denote a bipartite graph $G$ with color classes $A$ and $B$
by $G=(A,B,E)$. We use the same notation for a cobipartite
graph where $A$ and $B$ are the color classes of the complement.
A star is a bipartite graph $G=(A,B,E)$ with $|A|=1$. The single
vertex in
$A$ is called the center of the star.

In this section we show that there is a good characterization of
the cobipartite graphs that have a triangle partition.

\begin{theorem}
\label{cobipartite}
Let $G=(A,B,E)$ be cobipartite.
Then $G$ can be partitioned into
triangles if and only if one of the following holds true.
\begin{enumerate}[\rm (i)]
\item $|A| \bmod 3 = |B| \bmod 3=0$, or
\item
$|A| \bmod 3 =1$ and $|B| \bmod 3 =2$
and $G$ has a triangle with one vertex in $A$ and two in $B$, or
\item similar as above with the role of $A$ and $B$ interchanged.
\end{enumerate}
\end{theorem}
\begin{proof}
If $|A| \bmod 3=|B| \bmod 3=0$ then $G$ can be partitioned into
triangles. Assume that $|A| \bmod 3 =1$ and that $|B| \bmod 3 =2$.
If there exists a triangle with one vertex in $A$ and two in $B$
then $G$ can be partitioned into triangles.
For the converse,
assume that there is
no such triangle.
Let $G^{\prime}$ be the bipartite graph
obtained from $G$ by deleting edges between vertices that are contained
in the same color class.
Then $G^{\prime}$ does not contain a $P_3$,
{\em i.e.\/}, a path with three vertices, with
its midpoint in $A$.
Then $G^{\prime}$ has no $P_4$ and no $C_4$ and so $G^{\prime}$ is
trivially perfect~\cite{kn:wolk}. It follows that $G^{\prime}$ is a disjoint
collection of isolated vertices in $A$ and
stars with their centers in $B$. Then each triangle of $G$
that is not contained in $A$ nor in $B$ has one vertex in $B$ and two
in $A$. If $G$ has a triangle partition
then $|A| \bmod 3 = 2(|B| \bmod 3)$, which is a contradiction.
\qed\end{proof}

\begin{theorem}
There exists a linear-time algorithm which check if a cobipartite
graph has a triangle partition.
\end{theorem}
\begin{proof}
Assume that $|A| \mod 3 =1$ and that $|B| \mod 3 =2$.
It is easy to check in linear time whether $G^{\prime}$ is
disjoint collection of isolated vertices and
stars with their midpoints in $B$.
By Theorem~\ref{cobipartite} the graph $G$ has a partition into
triangles if and only if $G^{\prime}$ is not a disjoint
collection of isolated vertices in $A$ and stars with midpoints in $B$.
\qed\end{proof}

\section{Complements of multipartite graphs}

A $k$-partite graph is a graph $G=(V,E)$ of which the vertices
can be partitioned into $k$ independent sets. Notice that the
recognition of $3$-partite graphs is NP-complete since it is
an instance of the $3$-coloring problem. Henceforth, we assume
that the partition into color classes of a $k$-partite graph is a
part of the input.

In this section we show that
the triangle partition
problem on the complements of
$k$-partite graphs can be solved in polynomial time.
We start with the case where $k=3$.

\begin{lemma}
\label{3-partite}
Let $G$ be the complement of a  $3$-partite graph with color classes
$A_1$, $A_2$ and $A_3$. There exists a collection
of colored graphs $H_1,\ldots,H_t$, each with at most $42$ vertices,
such that $G$ has a partition into triangles if and only if one of the
graphs $H_i$ is an induced subgraph of $G$.
\end{lemma}
\begin{proof}
Consider a partition of $G$ into triangles. Each triangle is either
contained in one of the color classes or, it has one vertex in each
color class or,
it has one vertex in one color class and two vertices in another
color class. Assume that 3 vertices in $A_1$, $A_2$ and $A_3$ are
mutually connected by triangles. Then we can change the triangle partition
such that it contains the triangles on the three vertices in each $A_i$
instead. Assume that there are 3 vertices in $A_1$ that are in
triangles with 3 edges in $A_2$. Then we may replace those triangles
by one triangle in $A_1$ and two in $A_2$. In this way we obtain a
partition into triangles such that there are at most 14 vertices in each
color class $A_i$ that are in triangles that are not completely
contained in $A_i$. Consider all possible colored graphs
$H_i$ with at most $42$ vertices. Then there is a partition
of $G$ into triangles if and only if there is a colored induced subgraph
$H_i$, which can be partitioned into triangles, such that the
number of remaining vertices in each class $A_i$ is $0 \bmod 3$.
\qed\end{proof}

\begin{corollary}
There exists a polynomial-time algorithm that checks whether
the vertices of the complement of a $3$-partite graph can be partitioned
into triangles.
\end{corollary}

\begin{remark}
It is folklore that the triangle partition problem
is NP-complete on $3$-partite graphs, see,
{\em e.g.\/},~\cite{kn:morandini}.
\end{remark}

\begin{theorem}
\label{k-partite}
Let $k$ be a natural number. There exists a  polynomial-time
algorithm that solves the
triangle partition problem on complements of
$k$-partite graphs.
\end{theorem}
\begin{proof}
Let $G=(V,E)$ be the complement of a $k$-partite graph with color
classes $A_1,\ldots,A_k$. Assume that $V$ can be partitioned into
triangles. First we show that there exists a triangle partition
of $G$ with only a constant number of triangles of which the vertices
are not monochromatic. Consider three color classes $A_1$, $A_2$ and
$A_3$. By Lemma~\ref{3-partite} we may assume that at most
$14$ triangles that are not monochromatic are contained in $A_1+A_2+A_3$.
Since this holds for any three color classes, we find that there is a
partition with at most $14 k^3$ non-monochromatic triangles.
\qed\end{proof}

\section{Triangle packing on distance-hereditary graphs}

A graph $G$ is distance hereditary if for every component
in every induced subgraph the distance between two vertices
is the same as their distance in $G$~\cite{kn:howorka}.
In this section we show that the triangle packing problem
can be solved in polynomial time on distance-hereditary graphs.

A decomposition tree for a graph $G=(V,E)$ is a pair $(T,f)$ where
$T$ is a ternary tree and where $f$ is a 1-1 map from the vertices
in $G$ to the leaves of $T$. A line in $T$ induces a partition of
$V$ into two sets, say $A$ and $B$. The twinset
of $A$ is the subset of vertices in $A$ that have neighbors
in $B$. The graph $G$ is distance
hereditary if and only if it has a decomposition tree $(T,f)$
such that for every partition $\{A,B\}$ induced by a line in $T$
every pair of vertices in the twinset of $A$ have the same neighbors
in $B$~\cite{kn:oum}. If $G$ is distance hereditary, such a decomposition
tree for $G$ can be found in linear time~\cite{kn:dahlhaus2}.

\begin{theorem}
\label{DH}
There exists a polynomial-time
algorithm that solves the triangle
packing problem on distance-hereditary graphs.
\end{theorem}
\begin{proof}
Our method resembles the one used in~\cite{kn:guruswami} used to
solve the triangle packing problem on cographs.

Let $G=(V,E)$ be distance hereditary and let $(T,f)$ be a decomposition
tree for $G$ which satisfies the properties mentioned above. The
algorithm performs dynamic programming on branches of $T$ of
increasing size. Consider a branch $B$ rooted at some line $e$ of $T$.
Suppose that $B$ decomposes into two smaller branches $B_1$ and $B_2$.
Let $S_1$ and $S_2$ be the twinsets of the vertices
mapped to leaves of $B_1$ and $B_2$.
Then every vertex of $S_1$ is adjacent to every vertex of $S_2$
or no vertex of $S_1$ is adjacent to a vertex of $S_2$. Moreover,
the twinset for $B$ is either $S_1+S_2$ or,
it is one of the two or it is empty.

The dynamic programming keeps track of the maximum cardinality of a
triangle packing in a branch $B$ that avoids
a specified number of `free' vertices and a specified number
of `free' edges in the twinset. As an invariant,
the free edges do not contain any of the
free vertices and they form a matching. Free vertices and edges
can be used to form triangles with vertices outside the branch.
It is easy to update the table for a branch
$B$ from the tables of $B_1$ and $B_2$. Details for the updating procedure
can be found in~\cite{kn:guruswami}.
Since the number of entries of each table is bounded by $O(n^2)$, it follows
that the algorithm can be implemented to run in
polynomial time~\cite{kn:guruswami}.
\qed\end{proof}

\begin{remark}
It is fairly easy to see that the algorithm above can be
extended so that it works for graphs of bounded rankwidth~\cite{kn:oum}.
We have not been able to formulate the triangle partition problem
in monadic second-order logic. Guruswami, {\em et al.\/}, describe
an algorithm for the $K_r$-packing problem on cographs that runs in
polynomial time for each fixed $r$. The problem whether this
problem can be solved by a fixed-parameter algorithm with respect to
$r$ remains an open problem.
\end{remark}

\section{Modular permutations}

Let $\pi$ be a permutation of $\{1,\ldots,n\}$. A module in $\pi$
is a consecutive subsequence of $\pi$ that is a permutation
of a consecutive subsequence of $[1,\ldots,n]$.

Let $G$ be a permutation graph. Consider a diagram for $G$
with the labels $[1,\ldots,n]$ in order on the topline and with the
labels $[\pi(1),\ldots,\pi(n)]$  in order on the bottom line.
Then a module in the permutation corresponds with a subset
$M$ of vertices such that every vertex outside $M$ is adjacent to all
vertices of $M$ or to no vertex of $M$.

A subset $M$ of vertices in a graph $G$ with this property is called
a module of $G$. A module $M$ is trivial if it contains zero, one
or all the vertices of the graph. A module is strong if it does
not overlap with other modules. A graph is prime if it
contains only trivial modules. If $M \neq V$ is a strong module then there
exists a unique strong module $M^{\prime}$ of minimal size
that properly contains $M$. This defines a parent relation in a modular
decomposition tree for $G$.

A modular decomposition tree
is a rooted tree $T$ with a 1-1 map from the
leaves to a set $V$ of vertices. Each internal node of $T$ is
labeled as a join node, as a union node, or as a prime node.
A modular decomposition tree $T$
defines a graph $G$ with vertex set $V$ as follows.
A join node $j$ stands for the operation which adds an edge between
every pair of vertices that are mapped to leaves in different
subtrees of of $j$. A union node $u$ stands for the operation
that unions the subgraphs represented by the children of $u$.
Each prime node $p$ is labeled with a graph $H_p$. Each vertex
in $H_p$ corresponds with one child of $p$. If two vertices in $H_p$
are connected by an edge then every pair of vertices in the two
graphs represented by the two corresponding children is connected by an
edge.

Given a graph $G$, a tree that decomposes $G$ recursively into
strong modules
can be constructed in linear time~\cite{kn:tedder}.

\begin{definition}
A graph is {\em $k$-modular\/} if it has a modular decomposition
tree such that the graph $H_p$ of every prime node $p$ has at most
$k$ vertices.
\end{definition}

If the modular decomposition tree has no prime nodes
then the graph is decomposable by unions and joins. We call these
graphs $0$-modular. The class of $0$-modular graphs is exactly the
class of cographs.

\begin{theorem}
The graphs that are $k$-modular are characterized by a finite collection
of forbidden induced subgraphs.
\end{theorem}
\begin{proof}
This follows from Kruskal's theorem~\cite{kn:kruskal}.
\qed\end{proof}

Consider the class of graphs obtained from paths by replacing
the endvertices by false twins, {\em i.e.\/}, modules consisting of
two nonadjacent vertices. This class is contained in the class of
permutation graphs and it is not well-quasi-ordered by the induced
subgraph relation. This proves the following corollary.
Another way to see that is by showing that paths are prime. 

\begin{corollary}
There exists a function $f(k)$ such that $k$-modular permutation
graphs have no induced paths of length more than $f(k)$.
\end{corollary}

We omit the easy proof of the following theorem.

\begin{theorem}
For each natural number $k$ there exists a polynomial-time
algorithm which computes a triangle packing in $k$-modular graphs.
\end{theorem}

\section{Concluding remark}

The main question that we leave open in this paper is
whether there exists a polynomial-time algorithm that checks if
a permutation of $\{1,\ldots,3n\}$ can be partitioned into
increasing subsequences of length three.

\section{Acknowledgement}

We thank Klaas Zwartenkot for doing some calculations on the
number of permutations that can be partitioned into triangles.

An obvious lowerbound for the number of permutations of $\{1,\ldots,3n\}$ 
that can be partitioned into triangles is 
$\frac{(3n)!}{6^n}$. We have not been able 
to determine an exact formula for the number of permutations that can be 
partitioned into triangles nor have we been able to determine the 
asymptotics for those permutations.

\appendix

\section{Preliminaries on triangle packings}
\label{prel TP}

Despite great interest in the cycle packing -- and cycle cover
problem there is relatively little
theoretical progress on the first of the two problems.

The two kinds of problems are related by the
Erd\"os and P\'osa theorem which states that there is a function
$f(k) = O(k \log k)$ such that any graph contains either $k$
vertex-disjoint cycles or a set with $f(k)$ vertices which intersect
every cycle~\cite{kn:erdos}.

A graph $H$ is topologically contained in a graph $G$ if $G$ has a
subgraph $H^{\prime}$ which is a subdivision of $H$.
It is well-known that the graphs that have
no $k$ disjoint cycles are well-quasi ordered by topological
containment~\cite{kn:mader}.
See also~\cite[Theorem~5.6]{kn:graham}.
It follows that there exists a finite set $\mathcal{F}_k$ of graphs,
each containing a maximal number of $k$ disjoint cycles,
such that a graph $G$ has $k$ disjoint cycles if and only if
some element of $\mathcal{F}_k$ is topologically contained in it.
This can also be seen
as follows.
Notice that a graph $G$ does not have $k$ vertex-disjoint cycles
if and only if $G$ does not contain the graph $H$ that consists of
$k$ disjoint triangles as a minor. This implies
the previous observation~\cite[Proposition~5.16]{kn:graham}.
It follows also that the class of graphs
without $k$ disjoint cycles is minor closed. Furthermore,
this class does
not contain all planar graphs, so the class
has a uniform bound on the treewidth~\cite{kn:robertson2}.
Bodlaender subsequently
showed that the elements of this class can be recognized in
$O(n)$ time~\cite{kn:bodlaender}.
Likewise, for any natural number $k$
one can check in $O(n)$ time whether a graph has a
cycle cover with at most $k$ vertices.

Another research area that is related to the topic of this paper is that
of finding vertex colorings of graphs that are
restricted in some way. There are too many variations
to cover in any limited survey,
even when restricted
to permutation graphs.
We mention some
of the results that seem closely related to our research.

On partially
ordered sets, one of the major
contributions is the result of Greene and
Kleitman and of Frank~\cite{kn:frank,kn:greene2,kn:greene}.
Greene and Kleitman generalize Dilworth's theorem~\cite{kn:dilworth}
and Frank describes an efficient algorithm that finds an optimal
solution.
If $P=(V,\leq)$ is a
partially ordered set then one can find in polynomial time a collection
of $t$ antichains $\{A_1,\ldots,A_t\}$ that maximizes $|\cup_i A_i|$.
It follows that one can find in polynomial time an induced subgraph
of a permutation graph with a maximal number of vertices that has
chromatic number at most $t$. Likewise, one can find a collection
of $t$ cliques $\{C_1,\ldots,C_t\}$ that maximizes $|\cup_i C_i|$.

When one bounds the number of vertices in the color classes the picture
changes drastically.
The following problem has been investigated in great detail
due to its applications in various scheduling problems.
Suppose we wish to
find a vertex coloring with a minimal number of colors, such that
each color class contains at most $q$
vertices~\cite{kn:hansen,kn:hell,kn:kaller,kn:lonc,kn:moonen,kn:moonen2}.
This problem is NP-complete on permutation
graphs for each $q \geq 6$~\cite{kn:jansen}.

Suppose one wishes to
color the vertices of a graph with a minimal
number of colors such that each color
class induces a clique or an independent set. We call this
a homogeneous coloring.
The Erd\"os-Hajnal
conjecture states that for every graph $H$ there exists a $\delta < 1$
such that every graph $G$ that does not contain $H$ as an induced subgraph
can be homogeneously colored with at
most $n^{\delta} \log n$ colors~\cite{kn:alon,kn:erdos2}. Considerable
progress towards proving this conjecture is reported in~\cite{kn:fox}.
The conjecture is known to be true
for perfect graphs with $\delta=\frac{1}{2}$
and recently it was proved for bull-free
graphs
with $\delta=\frac{1}{4}$~%
\cite{kn:chudnovsky}.
One of the smallest graphs $H$ for which the conjecture
is still open is $C_5$.
Wagner showed that finding the minimum number of colors in a coloring
of this type
is NP-complete, even for permutation graphs~\cite{kn:wagner}.
On the other hand, for any pair of nonnegative numbers $r$ and $s$,
the class of permutation graphs that have a homogeneous coloring with
$r$ cliques and $s$ independent sets (possibly empty) is characterized
by a finite collection of forbidden induced subgraphs~\cite{kn:kezdy}.

Lonc mentions
the following open problem in~\cite{kn:lonc2}. Given a sequence of
$3n$ distinct positive integers.
Find a partition of the sequence into $n$ increasing subsequences,
each of 3 terms. This is equivalent to finding a partition of
the vertices of a permutation graph into triangles.

Packing triangles in a graph is NP-complete~\cite{kn:bbaker,kn:hell},
even
when restricted to
chordal graphs, planar graphs and linegraphs~\cite{kn:guruswami}.
Interestingly, the question whether the vertices of a
chordal graph can be partitioned into triangles
can be solved in polynomial time~\cite{kn:dahlhaus}.
Packing triangles
in splitgraphs, unit interval graphs and cographs
can be solved in polynomial
time~\cite{kn:dahlhaus,kn:guruswami,kn:manic}. For $r \geq 4$
the $K_r$-packing problem is NP-complete for 
splitgraphs~\cite{kn:guruswami}. 

When one allows
besides triangles also edges in the packing and one wishes
to maximize the number of vertices that are covered,
then this packing problem
becomes polynomial~\cite{kn:cornuejols,kn:hell2,kn:hell,kn:loebl}.

\section{Preliminaries on permutation graphs}
\label{prel perm}

A permutation diagram is obtained as follows.
Let $L_1$ and $L_2$ be two horizontal lines in the plane, one above
the other.
Label $n$ points on the topline and on the bottom line
by $1,2,\ldots,n$. Connect each point on the topline by a straight
linesegment
with the point with the identical label on the bottom line.
A graph is a permutation graph if it is the intersection graph
of the linesegments of a permutation diagram~\cite{kn:pnueli}.

If $G$ is a permutation graph then its complement $\Bar{G}$ is
also a permutation graph. This is easy to see; simply reverse the ordering
of the points on one of the two horizontal lines.
Also notice this: for any independent set the corresponding linesegments
are noncrossing. So they can be ordered left to right, which is
of course a transitive ordering. This shows that
$\Bar{G}$ is a comparability graph and so also $G$ is a
comparability graph. The converse holds as well since
transitive orderings of the vertices of $G$ and of $\Bar{G}$ provides
the ordering of the points on the top-- and bottom line. This can be seen
as follows. Let $F_1$ and $F_2$ be transitive orientations of $G$ and
$\Bar{G}$. We claim that $F_1+F_2$ is an acyclic orientation
of the complete graph. Otherwise there
is a directed triangle, and so two edges in the triangle
are directed according to
one of $F_1$ and $F_2$ and the third is directed according to the
other one of $F_1$ and $F_2$. But this contradicts the transitivity
of $F_1$ or the transitivity of $F_2$. Likewise, $F_1^{-1}+F_2$ is
acyclic. Order the vertices on the topline according to $F_1+F_2$ and
the vertices on the bottom line according to $F_1^{-1}+F_2$. It is
easy to check that this yields the permutation diagram.

\begin{theorem}[\cite{kn:dushnik}]
A graph $G$ is a permutation graph if and only if
$G$ and $\Bar{G}$ are comparability graphs.
\end{theorem}

The following characterization of permutation graphs illustrates
the relation of this class of graphs to the class of interval graphs.
Consider a collection of intervals on the real line. Construct a
graph of which the vertices are the intervals and make two
vertices adjacent if one of the two intervals contains the other.
Such a graph is called an interval containment graph.

Consider a diagram of a permutation graph. When one moves the bottom line
to the right of the topline then the linesegments in the diagram transform
into intervals. It is easy to check that two linesegments intersect if and
only if one of the intervals is contained in the other one. This proves that
permutation graphs are interval containment graphs. Now consider an interval
containment graph. Construct a permutation diagram as follows. Put the
left endpoints of the intervals in order
on the topline and the right endpoints
in order on the bottom line of the diagram.
Then one interval is contained in
another interval if and only if the two linesegments intersect. This
proves that every interval containment graph is a permutation graph.

\begin{theorem}[\cite{kn:dushnik}]
A graph is a permutation graph if and only if it is
an interval containment graph.
\end{theorem}

Permutation graphs can be recognized in linear
time~\cite{kn:mcconnell,kn:tedder}.
Notice that permutation graphs are perfect since they have no
induced cycles of length more than four~\cite{kn:chudnovsky2}.
A graph is perfect if for every induced subgraph the clique number
is the same as the chromatic number. If the clique -- and chromatic
number of a graph are the same then these numbers can be computed in
polynomial time~\cite{kn:grotschel}.
However, in a permutation graph computing
a largest clique corresponds to finding a
longest increasing subsequence and this can be computed
very efficiently~\cite{kn:fredman}. (Notice that a permutation graph has
a clique or an independent set with at least $\sqrt{n}$
vertices~\cite{kn:erdos3}.)


\begin{thebibliography}{99}

\bibitem{kn:alon}Alon,~N., J.~Pach and J.~Solymosi,
Ramsey-type theorems with forbidden subgraphs,
{\em Combinatorica\/} {\bf 21} (2001), pp.~155--170.

\bibitem{kn:bbaker}Baker,~B. and E.~Coffman,
Mutual exclusion scheduling,
{\em Theoretical Computer Science\/} {\bf 162} (1996), pp.~225--243.

\bibitem{kn:baker}Baker,~K., P.~Fishburn, F.~Roberts,
Partial orders of dimension 2,
{\em Networks\/} {\bf 2} (1971), pp.~11--28.

\bibitem{kn:bodlaender}Bodlaender,~H.,
On disjoint cycles,
{\em International Journal on Foundations of Computer Science\/} {\bf 5}
(1994), pp.~59--68.

\bibitem{kn:bodlaender2}Bodlaender,~H. and K.~Jansen,
Restrictions of graph partition problems. Part I,
{\em Theoretical Computer Science\/} {\bf 148} (1995), pp.~93--109.

\bibitem{kn:chudnovsky}Chudnovsky,~M. and S.~Safra,
The Erd\"os-Hajnal conjecture for bull-free graphs,
{\em Journal of Combinatorial Theory, Series B\/} {\bf 98} (2008),
pp.~1301--1310.

\bibitem{kn:chudnovsky2}Chudnovsky,~M., N.~Robertson, P.~Seymour
and R.~Thomas,
The strong perfect graph theorem,
{\em Annals of Mathematics\/} {\bf 164} (2006), pp.~51--229.

\bibitem{kn:cornelsen}Cornelsen,~S. and G.~Di~Stefano,
Treelike comparability graphs,
{\em Discrete Applied Mathematics\/} {\bf 157} (2009), pp.~1711--1722.

\bibitem{kn:cornuejols}Cornu\'ejols,~G., D.~Hartvigsen and W.~Pulleyblank,
Packing subgraphs in a graph,
{\em Operations Research Letters\/} {\bf 1} (1982), pp.~139--143.

\bibitem{kn:dahlhaus2}Dahlhaus,~E.,
Efficient parallel and linear time sequential split decomposition
(extended abstract),
{\em Proceedings of the $14^{\mathrm{th}}$ Conference on Foundations
of Software Technology and Theoretical Computer Science\/},
Springer LNCS~880 (1994), pp.~171--180.

\bibitem{kn:dahlhaus}Dahlhaus,~E. and M.~Karpinsky,
Matching and multidimensional matching in chordal and strongly
chordal graphs,
{\em Discrete Applied Mathematics\/} {\bf 84} (1998), pp.~79--91.

\bibitem{kn:dilworth}Dilworth,~R.,
A decomposition theorem for partially ordered sets,
{\em Annals of Mathematics\/} {\bf 51} (1950), pp.~161--166.

\bibitem{kn:dushnik}Dushnik,~B. and E.~Miller,
Partially ordered sets,
{\em American Journal of Mathematics\/} {\bf 63} (1941), pp.~600--610.

\bibitem{kn:erdos2}Erd\"os,~P. and A.~Hajnal,
Ramsey-type theorems,
{\em Discrete Applied Mathematics\/} {\bf 25} (1989), pp.~37--52.

\bibitem{kn:erdos}Erd\"os,~P. and L.~P\'osa,
On independent circuits contained in a graph,
{\em Canadian Journal of Mathematics\/} {\bf 17} (1965), pp.~347--352.

\bibitem{kn:erdos3}Erd\"os,~P. and G.~Szekeres,
A combinatorial problem in geometry,
{\em Compositio Mathematica\/} {\bf 2} (1935), pp.~463--470.

\bibitem{kn:felsner}Felsner,~S. and L.~Wernisch,
Maximum $k$-chains in planar point sets,
{\em SIAM Journal on Computing\/} {\bf 28} (1993), pp.~192--209.

\bibitem{kn:fox}Fox,~J. and B.~Sudakov,
Induced Ramsey-type theorems,
{\em Advances in Mathematics\/} {\bf 219} (2008), pp.~1771--1800.

\bibitem{kn:frank}Frank,~A.,
On chain and antichain families of a partially ordered set,
{\em Journal of Combinatorial Theory, Series B\/} {\bf 29} (1980),
pp.~176--184.

\bibitem{kn:fredman}Fredman,~M.,
On computing the length of longest increasing subsequences,
{\em Discrete Mathematics\/} {\bf 11} (1975), pp.~29--35.

\bibitem{kn:fulkerson}Fulkerson,~D. and O.~Gross,
Incidence matrices and interval graphs,
{\em Pacific Journal of Mathematics\/} {\bf 15} (1965), pp.~835--855.

\bibitem{kn:golumbic2}Golumbic,~M.,
{\em Algorithmic graph theory and perfect graphs\/},
Elsevier series {\em Annals of Discrete Mathematics\/} {\bf 57},
2004.

\bibitem{kn:graham}Graham,~R., M.~Gr\"otschel and L.~Lov\'asz (eds.),
{\em Handbook of Combinatorics, Volume~1\/},
North-Holland, Elsevier Science~B.~V., 1995.

\bibitem{kn:greene2}Greene,~C.,
Some partitions associated with a partially ordered set,
{\em Journal of Combinatorial Theory, Series A\/} {\bf 20} (1976),
pp.~69--79.

\bibitem{kn:greene}Greene,~C. and D.~Kleitman,
The structure of Sperner $k$-families,
{\em Journal of Combinatorial Theory, Series A\/}, {\bf 20} (1976),
pp.~41--68.

\bibitem{kn:grotschel}Gr\"otschel,~M., L.~Lov\'asz and A.~Schrijver,
Polynomial algorithms for perfect graphs,
{\em Annals of Discrete Mathematics\/} {\bf 21} (1984), pp.~325--356.

\bibitem{kn:guruswami2}Guruswami,~V. and C.~Rangan,
Algorithmic aspects of clique-transversal and clique-independent sets,
{\em Discrete Applied Mathematics\/} {\bf 100} (2000), pp.~183--202.

\bibitem{kn:guruswami}Guruswami,~V., C.~Rangan, M.~Chang,
G.~Chang and C.~Wong,
The $K_r$-packing problem,
{\em Computing\/} {\bf 66}, (2001), pp.~79--89.

\bibitem{kn:hajos}Haj\'os,~G.,
\"Uber eine Art von Graphen,
{\em Internationale Mathematische Nachrichten\/} {\bf 11} (1957),
Problem~65.

\bibitem{kn:hansen}Hansen,~P., A.~Hertz and J.~Kuplinsky,
Bounded vertex colorings of graphs,
{\em Discrete Mathematics\/} {\bf 111} (1993), pp.~305--312.

\bibitem{kn:hell2}Hell,~P.,
Graph packings,
{\em Electronic Notes in Discrete Mathematics\/} {\bf 5} (2000),
pp.~170--173.

\bibitem{kn:hell}Hell,~P. and D.~Kirkpatrick,
Packing by cliques and by finite families of graphs,
{\em Discrete Mathematics\/} {\bf 49} (1984), pp.~45--59.

\bibitem{kn:howorka}Howorka,~E.,
A characterization of distance-hereditary graphs,
{\em The Quarterly Journal of Mathematics\/} {\bf 28} (1977),
pp.~417--420.

\bibitem{kn:jansen}Jansen,~K.,
The mutual exclusion scheduling problem for permutation and
comparability graphs,
{\em Information and Computation\/} {\bf 180} (2003), pp.~71--81.

\bibitem{kn:kaller}Kaller,~D., A.~Gupta and T.~Shermer,
Linear-time algorithms for partial $k$-tree complements,
{\em Algorithmica\/} {\bf 27} (2000), pp.~254--274.

\bibitem{kn:kellog}Kellog,~S. and J.~Johnson,
Dominating sets in chordal graphs,
{\em SIAM Journal on Computing\/} {\bf 11} (1982), pp.~191--199.

\bibitem{kn:kezdy}K\'ezdy,~A., H.~Snevily and C.~Wang,
Partitioning permutations into increasing and decreasing subsequences,
{\em Journal of Combinatorial Theory, Series A\/} {\bf 73} (1996),
pp.~353--359.

\bibitem{kn:kruskal}Kruskal,~J.,
Well-quasi-ordering, the tree theorem, and Vazsonyi's conjecture,
{\em Transactions of the American Mathematical Society\/} {\bf 95}
(1960), pp.~210--225.

\bibitem{kn:liang}Liang,~Y. and C.~Rhee,
Finding a maximum matching in a circular-arc graph,
{\em Information Processing Letters\/} {\bf 45} (1993), pp.~185--190.

\bibitem{kn:loebl}Loebl,~M. and S.~Poljak,
Efficient subgraph packing,
{\em Journal of Combinatorial Theory, Series B\/} {\bf 59} (1993),
pp.~106--121.

\bibitem{kn:lonc2}Lonc,~Z.,
On complexity of chain and antichain partition problem,
{\em Proceedings WG'92\/}, Springer LNCS~509 (1992), pp.~97--104.

\bibitem{kn:lonc}Lonc,~Z.,
Chain partitions of ordered sets,
{\em Order\/} {\bf 11} (1994), pp.~343--351.

\bibitem{kn:lovasz}Lov\'asz,~L.,
On graphs not containing independent circuits (Hungarian),
{\em Mat. Lapok\/} {\bf 16} (1965), pp.~289--299.

\bibitem{kn:mader}Mader,~W.,
Wohlquasigeordnete Klassen endlicher Graphen,
{\em Journal of Combinatorial Theory, Series B\/} {\bf 12} (1972),
pp.~105--122.

\bibitem{kn:manic}Mani\'c,~G. and Y.~Wakabayashi,
Packing triangles in low degree graphs and indifference graphs,
{\em Discrete Mathematics\/} {\bf 308} (2008), pp.~1455--1471.

\bibitem{kn:mcconnell}McConnell,~R. and J.~Spinrad,
Modular decomposition and transitive orientation,
{\em Discrete Mathematics\/} {\bf 201} (1999), pp.~189--241.

\bibitem{kn:mckee}McKee,~T.~A. and F.~R.~McMorris,
{\em Topics in intersection graph theory\/},
SIAM Monographs on Discrete Mathematics and Applications, 1999.

\bibitem{kn:moitra}Moitra,~A. and R.~Johnson,
A parallel algorithm for maximum matching in interval graphs,
{\em Proceedings of the 1989 International Conference on
Parallel Processing\/} Vol.~III: Algorithms and Applications
(1989), pp.~114--120.

\bibitem{kn:moonen}Moonen,~L. and F.~Spieksma,
Partitioning a permutation graph: algorithms and an application.
Research report OR~0358, Katholieke Universiteit Leuven,
Belgium, 2003.

\bibitem{kn:moonen2}Moonen,~L. and F.~Spieksma,
Partitioning a weighted partial order,
{\em Journal of Combinatorial Optimization\/} {\bf 15} (2008),
pp.~342--356.

\bibitem{kn:morandini}Morandini,~M.,
NP-complete problem: partition into triangles.
Report~112360, Universit\`a di Udini, 2004.

\bibitem{kn:mucha}Mucha,~M. and P.~Sankowski,
Maximum matchings via Gaussian elimination,
{\em Proceedings of the $45^{\mathrm{th}}$ Annual IEEE Symposium
on Foundations of Computer Science\/} (2004), pp.~248--255.

\bibitem{kn:oum}Oum,~S.,
{\em Graphs of bounded rank-width\/},
PhD thesis, Princeton University, 2005.

\bibitem{kn:papadimitriou}Papadimitriou,~C. and M.~Yannakakis,
Scheduling interval-ordered tasks,
{\em SIAM Journal on Computing\/} {\bf 8} (1979), pp.~405--409.

\bibitem{kn:pnueli}Pnueli,~A., A.~Lempel and S.~Even,
Transitive orientation of graphs and identification of
permutation graphs,
{\em Canadian Journal of Mathematics\/} {\bf 23} (1971),
pp.~160--175.

\bibitem{kn:rhee}Rhee,~C. and Y.~Liang,
Finding a maximum matching in a permutation graph,
{\em Acta Informatica\/} {\bf 32} (1995), pp.~779--792.

\bibitem{kn:robertson2}Robertson,~N. and P.~Seymour,
Graph minors. V. Excluding a planar graph,
{\em Journal of Combinatorial Theory, Series B\/} {\bf 41} (1986),
pp.~92--114.

\bibitem{kn:robertson}Robertson,~N. and P.~Seymour,
Graph minors. XIII. The disjoint path problem,
{\em Journal of Combinatorial Theory, Series B\/} {\bf 63}
(1995), pp.~65--110.

\bibitem{kn:spinrad}Spinrad,~J.~P.,
{\em Efficient graph representations\/},
AMS Fields Institute Monographs {\bf 19}, 2003.

\bibitem{kn:spinrad2}Spinrad,~J., A.~Brandst\"adt and L.~Stewart,
Bipartite permutation graphs,
{\em Discrete Applied Mathematics\/} {\bf 18} (1987), pp.~279--292.

\bibitem{kn:tedder}Tedder,~M., D.~Corneil, M.~Habib and C.~Paul,
Simpler linear-time modular decomposition via recursive
factorizing permutations,
{\em Proceedings ICALP'08\/}, Springer LNCS~5125 (2008), pp.~634--645.

\bibitem{kn:wagner}Wagner,~K.,
Monotonic coverings of finite sets,
{\em Elektronische Informationsverarbeitung und Kybernetik\/} {\bf 20}
(1984), pp.~633--639.

\bibitem{kn:wolk}Wolk,~E.,
A note on ``the comparability graph of a tree,''
{\em Proceedings of the American Mathematical Society\/}
{\bf 16} (1965),
pp.~17--20.

\end{thebibliography}
\end{document}